\documentclass[11pt]{article}
\usepackage[centertags]{amsmath}
\usepackage{amsfonts}
\usepackage{amssymb}
\usepackage{amsthm,color}
\usepackage{newlfont}

\newtheorem{Theorem}{Theorem}[section]
\newtheorem{Definition}[Theorem]{Definition}

\newtheorem{Lemma}[Theorem]{Lemma}

\newtheorem{Remark}[Theorem]{Remark}
\newtheorem{Example}[Theorem]{Example}

\newtheorem{Hypothesis}{Hypothesis}

\numberwithin{equation}{section}

\begin{document}

\def\le{\left}
\def\r{\right}
\def\cost{\mbox{const}}
\def\a{\alpha}
\def\d{\delta}
\def\ph{\varphi}
\def\e{\epsilon}
\def\la{\lambda}
\def\si{\sigma}
\def\La{\Lambda}
\def\B{{\cal B}}
\def\A{{\mathcal A}}
\def\L{{\mathcal L}}
\def\O{{\mathcal O}}
\def\bO{\overline{{\mathcal O}}}
\def\F{{\mathcal F}}
\def\K{{\mathcal K}}
\def\H{{\mathcal H}}
\def\D{{\mathcal D}}
\def\C{{\mathcal C}}
\def\M{{\mathcal M}}
\def\N{{\mathcal N}}
\def\G{{\mathcal G}}
\def\T{{\mathcal T}}
\def\R{{\mathbb R}}
\def\I{{\mathcal I}}

\def\bw{\overline{W}}
\def\phin{\|\varphi\|_{0}}
\def\s0t{\sup_{t \in [0,T]}}
\def\lt{\lim_{t\rightarrow 0}}
\def\iot{\int_{0}^{t}}
\def\ioi{\int_0^{+\infty}}
\def\ds{\displaystyle}
\def\pag{\vfill\eject}
\def\fine{\par\vfill\supereject\end}
\def\acapo{\hfill\break}

\def\beq{\begin{equation}}
\def\eeq{\end{equation}}
\def\barr{\begin{array}}
\def\earr{\end{array}}
\def\vs{\vspace{.1mm}   \\}
\def\rd{\reals\,^{d}}
\def\rn{\reals\,^{n}}
\def\rr{\reals\,^{r}}
\def\bD{\overline{{\mathcal D}}}
\newcommand{\dimo}{\hfill \break {\bf Proof - }}
\newcommand{\nat}{\mathbb N}
\newcommand{\E}{\mathbb E}
\newcommand{\Pro}{\mathbb P}
\newcommand{\com}{{\scriptstyle \circ}}
\newcommand{\reals}{\mathbb R}

\def\Amu{{A_\mu}}
\def\Qmu{{Q_\mu}}
\def\Smu{{S_\mu}}
\def\H{{\mathcal{H}}}
\def\Im{{\textnormal{Im }}}
\def\Tr{{\textnormal{Tr}}}
\def\E{{\mathbb{E}}}
\def\P{{\mathbb{P}}}

\title{Smoluchowski-Kramers approximation and large deviations for infinite dimensional gradient systems }
\author{Sandra Cerrai, Michael Salins\\
\vspace{.1cm}\\
Department of Mathematics\\
 University of Maryland\\
College Park\\
 Maryland, USA}

\date{}

\maketitle

\begin{abstract}
  In this paper, we explicitly calculate the  quasi-potentials for the damped semilinear stochastic wave equation when the system is of gradient type.  We show that in this case the infimum of the quasi-potential with respect to all possible velocities does not depend on the density of the mass and does coincide with the quasi-potential of the corresponding stochastic heat equation that one obtains from the zero mass limit. This shows in particular  that the Smoluchowski-Kramers approximation can be used to approximate long time behavior in the zero noise limit, such as exit time and exit place from a basin of attraction.
\end{abstract}

\section{Introduction}
In the present paper, we consider the following damped wave equation in a bounded regular domain $\mathcal{O}\subset \mathbb{R}^d$, perturbed by noise
\begin{equation} \label{wave-eq}
\le\{  \begin{array}{l}
\ds{    \mu \frac{\partial^2 u^{\mu}_{\epsilon}}{\partial t^2} (t, \xi) = \Delta u^\mu_\e(t,\xi) - \frac{\partial u^{\mu}_{\epsilon}}{\partial t} (t,\xi) +B(u^\mu_\e(t,\cdot))(\xi) + \sqrt{\epsilon}\, \frac{\partial w^Q}{\partial t}(t,\xi), \ \ \ \ \xi \in \mathcal{O},\  t>0, }\\
\vs
\ds{ u^\mu_\e(0,\xi) = u_0(\xi),\ \ \ \
    \frac{\partial u^{\mu}_{\epsilon}}{\partial t}(0,\xi) = v_0(\xi),\ \ \xi \in\,\mathcal{O},\ \ \ \ \ u^\mu_\e(t, \xi) = 0,\ \  \xi \in \partial \mathcal{O},}
        \end{array}\r.
    \end{equation}
for some parameters $0<\e,\mu<<1$. Here, $\partial w^Q/\partial t$ is a cylindrical Wiener process, white in time and colored in space, with covariance operator $Q^2$, for some $Q \in\,\mathcal{L}(L^2(\mathcal{O}))$.  Concerning the non-linearity $B$, we assume
\[B(x)=-Q^2 DF(x),\ \ \ \ x \in\,L^2(\mathcal{O}),\]
for some $F:L^2(\mathcal{O})\to \reals$, satisfying suitable conditions.
We also consider the semi-linear heat equation
\begin{equation} \label{heat-eq}
\le\{  \begin{array}{l}
\ds{  \frac{\partial u_\epsilon}{\partial t} (t,\xi) = \Delta u_\epsilon(t,\xi) + B(u_\epsilon(t,\cdot))(\xi) + \sqrt{\epsilon}\, \frac{\partial w^Q}{\partial t}(t,\xi), \ \ \ \  \xi \in \mathcal{O}, \ t>0,}\\
\vs
\ds{  u_\epsilon(0,\xi) = u_0(\xi),\ \ \ \ \ \ u_\epsilon(t,\xi) = 0, \xi \in \partial \mathcal{O}.}
  \end{array}\r.
\end{equation}

In \cite{smolu2}, it has been proved that  if $\epsilon>0$ is fixed and $\mu$ tends to zero, the solutions of \eqref{wave-eq} converge to the solution of \eqref{heat-eq}, uniformly on compact intervals. More precisely, for any $\eta>0$ and $T>0$
\[\lim_{\mu\to 0}\mathbb{P}\le(\sup_{t \in\,[0,T]}|u^\mu_\e(t)-u_\e(t)|_{L^2(\mathcal{O})}>\eta\r)=0.\]
Moreover, in the case $d=1$ and $Q=I$, it has been proven that for any fixed $\epsilon>0$, the first marginal of the invariant measure  of equation \eqref{wave-eq} coincides with the invariant measure of equation \eqref{heat-eq}, for any $\mu>0$.

In this paper, we are interested in comparing the small noise behavior of the two systems. More precisely, we keep $\mu>0$ fixed and let $\epsilon$ tend to zero, to study some relevant quantities associated with the large deviation principle for these systems, as the quasi-potential that describes also the asymptotic behavior of the expected exit time  from a domain and the corresponding exit places.  Due to the gradient structure of \eqref{wave-eq}, as in the finite dimensional case studied in \cite{f}, we are here able to calculate explicitly the quasi-potentials $V^\mu(x,y)$ for system \eqref{wave-eq}
as
\begin{equation}
\label{m2-i1}
   V^\mu(x,y) = \left|(-\Delta)^{\frac{1}{2}}Q^{-1} x \right|_{L^2(\mathcal{O})}^2 + 2F(x) + \mu \left|Q^{-1} y \right|_{L^2(\mathcal{O})}^2.\end{equation}
   Actually, we can prove that for any $\mu>0$
   \begin{equation}
   \label{m2-i2}
   V^{\mu}(x,y) = \inf \left\{ I^{\mu}_{-\infty}(z)\, :\, z(0) = (x,y),\  \lim_{t \to - \infty} \left|C_\mu^{-1/2} z(t) \right|_H = 0 \right\},\end{equation}
   where
   \[I^\mu_{-\infty}(z)=\frac 12 \int_{-\infty}^0\le|Q^{-1}\le( \mu \frac{\partial^2 \varphi}{\partial t^2} (t) - \Delta \varphi(t) + \frac{\partial \varphi}{\partial t} (t) -B(\varphi(t))\r)\r|_{L^2(\mathcal{O})}^2\,dt\]
   with $\varphi(t)=\Pi_1z(t)$, and
\[C_\mu(u,v)=\frac 12\le((-\Delta)^{-1}Q^2 u,\frac 1\mu(-\Delta)^{-1}Q^2v\r),\ \ \ \ (u,v) \in\,L^2(\mathcal{O})\times H^{-1}(\mathcal{O}).\]
From \eqref{m2-i2}, we obtain that
\[   V^\mu(x,y) \geq \left|(-\Delta)^{\frac{1}{2}}Q^{-1} x \right|_{L^2(\mathcal{O})}^2 + 2F(x) + \mu \left|Q^{-1} y \right|_{L^2(\mathcal{O})}^2.\]
Thus, we obtain the equality \eqref{m2-i1} by constructing a a path which realizes the minimum.
An immediate consequence of \eqref{m2-i1} is that for each $\mu>0$
\begin{equation}
\label{m2-i3}
V_\mu(x)=\inf_{y \in\,H^{-1}(\mathcal{O})}V^\mu(x,y)=V^\mu(x,0)=V(x),\end{equation}
where $V(x)$ is the quasi-potential associated with equation \eqref{heat-eq}.

Now, consider an open bounded domain $G\subset L^2(D)$,  which is invariant for $u^\mu_0$ and is attracted to the asymptotically stable equilibrium  $0$, and for any $x \in\,G$ let us define
\[\tau^{\mu}_\e:=\inf\,\le\{t\geq 0\,:\,u^{\mu}_\e(t) \in\,\partial G\,\r\} \text{ and } \tau_\epsilon = \inf\{t \geq 0: u_\epsilon(t) \in \partial G\}.\]
In a forthcoming paper we plan to prove
\begin{equation}
\label{intro22}
\lim_{\e\to 0}\,\e\log\,\E\tau^{\mu}_\e=\inf_{y \in\,\partial G} V_\mu(y)\end{equation}
As a consequence of \eqref{m2-i3}, this would imply that, in the gradient case
for any fixed $\mu>0$  it holds
\[\lim_{\e\to 0}\,\e\log\,\E\tau^{\mu}_\e=\inf_{y \in\,\partial G} V(y) = \lim_{\e \to 0}\,\e\log\,\E\tau^\mu_\e.\]

\section{Preliminaries and assumptions}

Let $H=L^2(\mathcal{O})$ and let $A$ be the realization of the Laplacian with Dirichlet boundary conditions in $H$.  Let $\{e_k\}_{k \in\,\nat}$ be the complete orthonormal basis of eigenvectors of $A$, and let $\{-\a_k\}_{k \in\,\nat}$ be the corresponding sequence of eigenvalues, with  $0<\a_1\leq  \alpha_k \leq \alpha_{k+1}$, for any $k \in\,\nat$.

The stochastic perturbation  is given  by a cylindrical Wiener process $w^{Q}(t,\xi)$, for $t\geq 0$ and $\xi \in\,{\cal O}$, which
is assumed to be white in time and colored in space, in the case
of space dimension $d>1$. Formally, it  is defined as the infinite sum
\begin{equation}
\label{noise3}
w^{Q}(t,\xi)=\sum_{k=1}^{+\infty} Q e_{k}(\xi)\,\beta_{k}(t),
\end{equation} where
 $\{e_{k}\}_{k \in\,\nat}$ is the
complete orthonormal basis in $L^2(\mathcal{O})$ which diagonalizes $A$ and  $\{\beta_{k}(t)\}_{k
\in\,\nat}$ is a sequence of mutually independent standard
Brownian motions defined on the same complete stochastic basis
$(\Omega,\mathcal{F}, \mathcal{F}_t, \mathbb{P})$.
\begin{Hypothesis}
\label{H1}
The linear operator $Q$ is bounded in $H$ and diagonal with respect to the basis $\{e_k\}_{k \in\,\nat}$ which diagonalizes $A$. Moreover, if $\{\la_k\}_{k \in\,\nat}$ is the corresponding sequence of eigenvalues, we have
\begin{equation}
\label{m1}
  \sum_{k=1}^\infty \frac{\lambda_k^2}{\alpha_k} < + \infty.
\end{equation}
\end{Hypothesis}
In particular, if $d=1$ we can take $Q=I$, but if $d>1$ the noise has to colored in space.
Concerning the non-linearity $B$, we assume that it has the following gradient structure.

\begin{Hypothesis}
\label{H2}
There exists $F : H \to \mathbb{R}$ of class $C^1$, with $F(0)=0$, $F(x) \geq 0$ and $\left<D\!F(x),x \right> \geq 0$ for  all $x \in\,H$, such that
\[B(x) = -Q^2 D\!F(x),\ \ \ \ x \in\,H.\]
Moreover,
\begin{equation} \label{m2-00}
  \left| D\!F(x) - D\!F(y) \right|_H \leq \kappa \left|x - y \right|_H,\ \ \ \ x, y \in\,H.
\end{equation}
\end{Hypothesis}

\begin{Example}
\em{ \begin{enumerate}
\item  Assume $d=1$ and take $Q=I$. Let $b: \mathbb{R} \to \mathbb{R}$ be a decreasing Lipschitz continuous function with $b(0)=0$.  Then the composition operator $B(x)(\xi) = b(x(\xi))$, $\xi \in\,\mathcal{O}$, is of gradient type.  Actually, if we set
  \begin{equation*}
    F(x) = -\int_\mathcal{O} \int_0^{x(\xi)} b(\eta) d \eta d \xi,\ \ \ \ x \in\,H,
  \end{equation*}
  we have
  \[B(x)=-D\!F(x),\ \ \ \ x \in\,H.\]
  Moreover,
 it is clear that $F(0)=0$, $F(x) \geq 0$ for all $x \in\,H$,  and
 \[\left<D\!F(x),x \right>=-\int_\mathcal{O} b(x(\xi))x(\xi)\,d\xi\geq 0,\ \ \  x \in\,H.\]
   \item
  Assume now $d\geq 1$, so that $Q$ is a general bounded operator in $H$, satisfying Hypothesis \ref{H1}.  Let $b: \mathbb{R} \to \mathbb{R}$ be a function of class $C^1$, with Lipschitz-continuous first derivative, such  that $b(0) = 0$ and $b(\eta) \geq 0$, for all $\eta \in\,\reals$. Moreover, the only local minimum of $b$ occurs at $0$.  Let
  \[F(x)= \int_\mathcal{O} b( x(\xi)) d\xi,\ \ \ x \in\,H.\]
It is immediate to check that $F(0)=0$ and $F(x)\geq 0$, for all $x \in\,H$.  Furthermore, for any $x \in\,H$
\[D\!F(x)(\xi)= b^\prime (x(\xi)),\ \ \ \  \xi \in\,\mathcal{O}.\]
 Therefore, the nonlinearity
 \[B(x) = -Q^2  b^\prime(x(\cdot)),\ \ \ x \in\,H,\]
satisfies Hypothesis \ref{H2}.
  \end{enumerate}}
\end{Example}

\medskip

For any $\delta \in \mathbb{R}$, we denote by $H^\delta$  the completion of $C_0^\infty(\mathcal{O})$ in the norm
\begin{equation*}
  |u|_{H^\delta}^2 = \sum_{k=1}^\infty \alpha_k^\delta \left<u, e_k \right>_H.
\end{equation*}
This is a Hilbert space, with the scalar product
\begin{equation*}
  \left< u,v \right>_{H^\delta} = \sum_{k=1}^\infty \alpha_k^\delta \left< u, e_k \right>_H \left< v, e_k \right>_H.
\end{equation*}
In what follows, we shall define $\H_\delta = H^{\delta} \times H^{\delta -1}$ and we shall set $\H = \H_0$.

Next, for $\mu>0$, we define $\Amu : D(\Amu) \subset \H_\delta \to \H_\delta$ by
\begin{equation}
  \Amu (u,v) = \left( v, \frac{1}{\mu} A u - \frac{1}{\mu}v \right), \, (u,v) \in D(\Amu) = \H_{\delta +1},
\end{equation}
and we denote by $\Smu(t)$ the semigroup on $\H$ generated by $\Amu$. \\In \cite[Proposition 2.4]{smolu2}, it has been proven that $\Smu(t)$ is a $C_0$-semigroup  of negative type, namely, there exist $M_\mu>0$ and $\omega_\mu>0$ such that
\begin{equation} \label{negative-type-eq}
  \|\Smu(t) \|_{L(\H_\delta)}  \leq M_\mu e^{-\omega_\mu t},\ \ \ \ t\geq 0.
\end{equation}
Moreover, for any $\mu>0$ we  define the operator $\Qmu : H^{\delta-1} \to \H_\delta$ by setting
\begin{equation*}
  \Qmu v = \frac 1\mu \left( 0 , Qv \right),\ \ \ v \in\,H^{\d-1}.
\end{equation*}
Therefore, if we define
\[\hat{F}(u,v)=F(u),\ \ \ \ \hat{Q}(u,v)=Q u,\ \ \ \ \ (u,v) \in\,\H,\]
equation \eqref{wave-eq} can be rewritten as the following abstract stochastic evolution equation in the space $\H$
\begin{equation}
\label{abstract-wave}
dz_\e^\mu(t)=\le[A_\mu z_\e^\mu(t)-Q_\mu \hat{Q}D\!\hat{F}(z_\e^\mu(t))\r]\,dt+\sqrt{\e}\,Q_\mu\,dw(t),\ \ \ \ \ z^\mu_\e(0)=(u_0,v_0).
\end{equation}
  Analogously, equation \eqref{heat-eq} can be rewritten as the following abstract stochastic evolution equation in $H$
\begin{equation}
\label{abstract-heat}
du_\e(t)=\le[A u_\e(t)-Q^2 D\!F(u_\e(t))\r]\,dt+\sqrt{\e}\,Qdw(t),\ \ \ \ u_\e(0)=u_0.
\end{equation}

\begin{Definition}
\begin{enumerate}
\item   A predictable process $z^{\mu}_{\epsilon} \in L^2(\Omega,C([0,T];\H))$ is a mild solution to \eqref{abstract-wave} if
\[    z^{\mu}_{\epsilon}(t) = \Smu(t)(u_0,v_0) -\int_0^t \Smu(t-s) Q_\mu \hat{Q} D\!\hat{F}(z^{\mu}_{\epsilon}(s)) ds + \sqrt{\epsilon} \int_0^t \Smu(t) \Qmu dw(s).
  \]
\item   A predictable process $u^\epsilon \in L^2(\Omega, C([0,T];H))$ is a mild solution to \eqref{abstract-heat} if
  \[
    u^\epsilon(t) = e^{tA} u_0 - \int_0^t e^{(t-s)A} Q^2 D\!F( u^\epsilon(s)) ds + \sqrt{\epsilon} \int_0^t e^{(t-s)A} Q dw(s).
  \]
  \end{enumerate}
\end{Definition}

\begin{Remark}
{\em If we define
\[C_\mu=\int_0^{+\infty} S_\mu(t)Q_\mu Q^\star_\mu S^\star_\mu(t)\,dt,\]
as shown in \cite[Proposition 5.1]{smolu2} we have
\[C_\mu(u,v)=\frac 12\le((-A)^{-1}Q^2 u,\frac 1\mu(-A)^{-1}Q^2v\r),\ \ \ \ (u,v) \in\,\H.\]
Therefore,
we get
\[
2A_\mu C_\mu\, D\!\hat{F}(u,v)=(0,-\frac 1\mu Q^2 D\!F(u))=-Q_\mu \hat{Q} D\!\hat{F}(u,v) ,\ \ \ \ (u,v) \in\,\H.
\]
This means that equation \eqref{abstract-wave} can be rewritten as
\[dz_\e^\mu(t)=\le[A_\mu z_\e^\mu(t)+2A_\mu C_\mu\, D\!\hat{F}(z_\e^\mu(t))\r]\,dt+\sqrt{\e}\,Q_\mu\,dw(t),\ \ \ \ \ z^\mu_\e(0)=(u_0,v_0).
\]
In the same way, equation \eqref{abstract-heat} can be rewritten as
\[
du_\e(t)=\le[A u_\e(t)+2A CD\!F(u_\e(t))\r]\,dt+\sqrt{\e}\,Qdw(t),\ \ \ \ u_\e(0)=u_0,
\]
where
\[
C=\int_0^{+\infty} e^{tA}QQ^\star e^{tA^\star}\,dt=\frac 12 (-A)^{-1}Q^2.
\]
In particular, both \eqref{abstract-wave} and \eqref{abstract-heat} are gradient systems (for more details see \cite{fur})}
\end{Remark}

As we are assuming that $D\!F:H\to H$ is Lipschitz continuous, we have that equation \eqref{abstract-wave} has a unique mild solution $z^{\mu}_{\epsilon} \in L^p(\Omega,C([0,T];\H))$ and  equation \eqref{abstract-heat} has a unique mild solution  $u^\epsilon \in L^p(\Omega, C([0,T];H))$, for any $p\geq 1$ and $T>0$.

In \cite[Theorem 4.6]{smolu2} it has been proved that in this case the so-called Smoluchowski-Kramers approximation holds. Namely, for any $\e, T>0$ and $\eta>0$
\[\lim_{\mu\to 0}\,\Pro\le(\sup_{t \in\,[0,T]}|u^\mu_\e(u)-u_\e(t)|_H>\eta\r)=0,\]
where $u^\mu_\e(t)=\Pi_1z^\mu_\e(t).$

\section{A characterization of the quasi-potential}

For any $\mu>0$ and $t_1<t_2$, and for any $z \in\,C([t_1,t_2];\H)$ and $z_0 \in\,\H$, we define
\begin{equation} \label{rate-functional-wave}
  I^\mu_{t_1,t_2}(z) = \frac{1}{2} \inf \left\{ \left| \psi \right|_{L^2((t_1,t_2);H)}^2 : z = z^\mu_{z_0,\psi} \right\}
\end{equation}
where $z^\mu_{\psi}$ solves the following skeleton equation associated with \eqref{abstract-wave}
\begin{align} \label{z-psi-control-eq}
  z^\mu_{z_0,\psi}(t) = \Smu(t - t_1) z_0 + \int_{t_1}^t \Smu(t-s) A_\mu C_\mu\,D\!\hat{F}_\mu(z^\mu_{z_0,\psi}(s)) ds+ \int_{t_1}^t \Smu(t-s) \Qmu \psi(s) ds.
\end{align}
Analogously,
for $t_1<t_2$, and for any $\varphi\in\,C([t_1,t_2];H)$ and $u_0 \in\,H$, we define
\begin{equation} \label{rate-functional-heat}
  I_{t_1,t_2} ( \varphi) = \frac{1}{2} \inf \left\{ \left| \psi \right|_{L^2((t_1,t_2);H)}^2 : \varphi = \varphi_{\psi,u_0} \right\},
\end{equation}
where $\varphi_{u_0,\psi}$ solves the  problem
\begin{equation} \label{varphi-psi-control-eq}
  \varphi_{u_0,\psi}(t) = e^{(t-t_1)A} u_0 - \int_{t_1}^t e^{(t-s)A} Q^2 D\!F(\varphi_{u_0,\psi}(s)) ds + \int_{t_1}^t e^{(t-s)A} Q \psi(s)ds.
\end{equation}
In what follows, we shall also denote
\[  I^\mu_{-\infty}(z) = \sup_{t<0} I^\mu_{t,0}(z),\ \ \ \ \ I_{-\infty}(\varphi) = \sup_{t<0} I_{t,0}(\varphi).
\end{equation*}

Since \eqref{abstract-wave} and \eqref{abstract-heat} have additive noise, as a consequence of the contraction lemma we have that the family $\{z^{\mu}_{\epsilon}\}_{\epsilon>0}$ satisfies the large deviations principle in $C([0,T];\H)$, with respect to the rate function $I^\mu_{0,T}$ and the family $\{u^\epsilon\}_{\epsilon>0}$ satisfies the large deviations principle in $C([0,T];H)$, with respect to the rate function $I_{0,T}$.

In what follows, for any fixed $\mu>0$ we shall denote by $V^\mu$  the quasi-potential associated with system \eqref{abstract-wave}, namely
\[   V^\mu(x,y) = \inf \left\{ I^\mu_{0,T}(z)\, :\, z(0)= 0,\  z(T) = (x,y),\  T>0 \right\}.
 \]
Analogously, we shall denote by $V$ the  quasi-potential associated with equation \eqref{abstract-heat}, that is
  \[  V(x) = \inf \left\{ I_{0,T}(\varphi)\,:\, \varphi(0) = 0,\  \varphi(T) = x,\  T>0 \right\}.\]
Moreover, for any $\mu>0$ we shall define
\[V_\mu(x)=\inf_{y \in\,H^{-1}(\mathcal{O})} V^\mu(x,y)=\inf \left\{ I^\mu_{0,T}(z)\,:\, z(0) = 0,\  \Pi_1 z(T) = x,\  T>0 \right\}.\]
In \cite[Proposition 5.4]{cerrok}, it has been proven   that $V(x)$ can be represented as
\begin{equation*}
  V(x) = \inf \left\{ I_{-\infty} (\varphi)\,:\, \varphi(0) = x,\  \lim_{t \to -\infty} \left|\varphi(t) \right|_{H} = 0 \right\}.
\end{equation*}
Here, we want   to prove a similar representations for $V^\mu(x,y)$, for any fixed $\mu>0$.
To this purpose, we first  introduce the operator $L^\mu_{t_1,t_2} : L^2((t_1,t_2):H) \to \H$, defined as
\begin{equation}
  L^\mu_{t_1,t_2} \psi = \int_{t_1}^{t_2} \Smu(t_2-s) \Qmu \psi(s) ds.
\end{equation}

\begin{Theorem}
  For any $\mu>0$ and $(x,y) \in \H$ we have
  \begin{equation}
    V^{\mu}(x,y) = \inf \left\{ I^{\mu}_{-\infty}(z)\, :\, z(0) = (x,y),\  \lim_{t \to - \infty} \left|C_\mu^{-1/2} z(t) \right|_H = 0 \right\}.
  \end{equation}
\end{Theorem}

\begin{proof}
  First we observe that by the definitions of $I^\mu_{t_1,t_2}$ and $V^\mu(x,y)$,
  \begin{equation*}
    V^\mu(x,y) = \inf \left\{I^\mu_{-T,0}(z): z(-T)=0, z(0) = (x,y), T>0 \right\}.
  \end{equation*}
Now, for any $\mu>0$ and $(x,y) \in\,\H$,  we define
  \begin{equation*}
    M^\mu(x,y):= \inf \left\{ I^{\mu}_{-\infty}(z)\, :\, z(0) = (x,y),\ \lim_{t \to - \infty} \left|C_\mu^{-1/2} z(t) \right|_\H = 0 \right\}.
  \end{equation*}
  Clearly, we want to prove that $M^\mu(x,y)=V^\mu(x,y)$, for all $(x,y) \in\,\H$.

  If $z$ is a continuous path with $z(-T)=0$ and $z(0) = (x,y)$, we can extend it  in $C((-\infty,0);\H)$, by defining $z(t) = 0$, for $t<-T$.  Then, since $D\!F(0)=0$,  we see that
  \[I^\mu_{-\infty}(z) = I^\mu_{-T,0}(z),\]
so that $M^\mu(x,y) \leq V^\mu(x,y)$.

Now, let us prove that the opposite inequality holds.  $M^\mu(x,y)= +\infty$, there is nothing else to prove.  Thus, we assume that $M^\mu(x,y) < +\infty$.  This means that for any $\e>0$ there must be some  $z^\mu_\e \in C((-\infty,0);\H)$, with the properties that $z^\mu_\e(0) = (x,y)$ and
\[\lim_{t \to -\infty} |C_\mu^{-1/2} z^\mu_\e(t) |_\H = 0,\ \ \ I^\mu_{-\infty}(z^\mu_\e) \leq M^\mu(x,y) + \epsilon.\]

  In what follows we shall prove that the following auxiliary result holds.
  \begin{Lemma} \label{L-inverse-lem}
  For any $\mu>0$, there exists  $T_\mu > 0$ such that for any $t_1<t_2-T_\mu$ we have $\text{Im}(L^\mu_{t_1,t_2}) = D(C_\mu^{-1/2})$ and
  \begin{equation}
    \left|(L^\mu_{t_1,t_2})^{-1} z \right|_{L^2((t_1,t_2);H)} \leq c(\mu, t_2 - t_1) \left| C_\mu^{-1/2} z \right|_\H,\ \ \ z \in\,\text{Im}(L^\mu_{t_1,t_2}),
  \end{equation}
  where
  \begin{equation*}
    C_\mu (x,y) = \left( (-A)^{-1} Q^2 x, \frac{1}{\mu} (-A)^{-1} Q^2 y \right).
  \end{equation*}
\end{Lemma}

 Thus,  let $T_\mu$ be the constant from Lemma \ref{L-inverse-lem} and let $t_\epsilon<0$ be such that
 \[|C_\mu^{- 1/2} z^\mu_\e(t_\epsilon)|_{\H} < \epsilon.\]
Moreover, let
  $\psi^\mu_\e := (L^\mu_{t_\epsilon - T_\mu, t_\epsilon})^{-1} z^\mu_\e(t_\epsilon)$.  By Lemma \ref{L-inverse-lem} we have
  \begin{equation*}
    \left| \psi^\mu_\e \right|_{L^2(t_\epsilon - T_\mu, t_\epsilon;H)} \leq c_\mu \epsilon.
  \end{equation*}
Next, we define
  \begin{equation*}
    \hat{z}^\mu_\e(t) =  \int_{t_\epsilon - T_\mu}^t \Smu(t-s) Q_\mu \psi^\mu_\e(s) ds.
  \end{equation*}
Thanks to  \eqref{negative-type-eq}, we have
  \begin{equation*}
  \begin{array}{l}
    \displaystyle{ \int_{t_\epsilon - T_\mu}^{t_\epsilon} |\hat{z}^\mu_\e(t) |_\H^2 dt
    \leq \int_{t_\epsilon - T_\mu}^{t_\epsilon} \left( \int_{t_\epsilon - T_\mu}^t \frac{M_\mu}{\mu} e^{-\omega_\mu (t-s)}
    |Q \psi^\mu_\e(s)|_{H^{-1}} ds \right)^2 dt}\\
    \vs
    \displaystyle{ \leq \frac{M_\mu^2}{2 \mu^2 \omega_\mu}\int_{t_\epsilon - T_\mu}^{t_\epsilon}  |Q \psi^\mu_\e|^2_{L^2((t_\epsilon - T_\mu, t\epsilon);H^{-1})} dt \leq T_\mu \frac{M_\mu^2}{2\mu^2 \omega_\mu} |Q \psi^\mu_\e|^2_{L^2((t_\epsilon - T_\mu, t\epsilon);H^{-1})}\leq c_\mu \e^2.}
    \end{array}
  \end{equation*}
Furthermore, $\hat{z}^\mu_\e(t_\epsilon - T_\mu) = 0$ and $\hat{z}^\mu_\e(t_\epsilon) = z(t_\epsilon)$.  Finally, we notice that
\[\begin{array}{l}
\ds{    \hat{z}^\mu_\e(t) = -\int_{t_\epsilon - T_\mu}^t \Smu(t-s) \Qmu Q D\!{F}(\hat{z}^\mu_\e(s)) ds+ \int_{t_\epsilon - T_\mu}^t \Smu(t-s) \Qmu \left( \psi^\mu_\e(s) + Q D\!{F}( \hat{z}^\mu_\e(s)) \right) ds,}
\end{array}
  \]
so that
  \begin{equation*}
    I^\mu_{t_\epsilon - T_\mu, t_\epsilon}(\hat{z}^\mu_\e) = \frac{1}{2} \left| \psi^\mu_\e + Q D\!{F}(\hat{z}^\mu_\e(s)) \right|_{L^2((t_\epsilon - T_\mu, t_\epsilon);H)}^2
    \leq c_\mu \epsilon^2.
  \end{equation*}
  Now if we define
  \begin{equation*}
    \tilde{z}^\mu_\e(t) = \begin{cases}
                    \hat{z}^\mu_\e(t) & \text{ if } t_\epsilon - T_\mu \leq t < t_\epsilon \\
                    z^\mu_\e(t) & \text{ if } t_\epsilon \leq t \leq 0,
                   \end{cases}
  \end{equation*}
  we see that $\tilde{z}^\mu_\e \in C((t_\epsilon - T_\mu,0);\H)$ and
  \begin{equation*}
    V^\mu(x,y) \leq I^\mu_{t_\epsilon - T_\mu,0}(\tilde{z}^\mu_\e) = I^\mu_{t_\epsilon - T_\mu, t_\epsilon}(\hat{z}^\mu_\e) + I^\mu_{t_\epsilon,0}(z^\mu_\epsilon) \leq c_\mu \epsilon^2 + M^\mu(x,y) + \epsilon.
  \end{equation*}
Due to the arbitrariness  of $\epsilon>0$, we can conclude.
\end{proof}

{\em Proof of Lemma \ref{L-inverse-lem}.} It is immediate to check that
  \begin{equation*}
    \left|(L^\mu_{t_1,t_2})^\star z \right|_{L^2((t_1,t_2);H)}^2 = \frac{1}{\mu} \int_0^{t_2-t_1} \left| Q_\mu^\star S_\mu^\star(s) z \right|_H^2 ds.
  \end{equation*}
  therefore, since
  \[Q_\mu^\star(u,v)=\frac 1\mu (-A)^{-1} Qv,\ \ \ (u,v) \in\,\H,\]
  (see \cite[Section 5]{smolu2} for a proof), we can conclude
  \begin{equation*}
    \left|(L^\mu_{t_1,t_2})^\star z \right|_{L^2((t_1,t_2);H)}^2 = \frac{1}{\mu^2} \int_0^{t_2-t_1} \left| Q(-A)^{-1}\Pi_2 S_\mu^\star(s) z \right|_H^2 ds.
  \end{equation*}
Now, if we expand  $S_\mu^\star(t)$ into Fourier coefficients, by \cite[Proposition 2.3]{smolu2}, we get
  \begin{equation*}
    S_\mu^\star (u,v) = \sum_{k=1}^\infty \left( \hat{f}_k^\mu(t) e_k, \hat{g}_k^\mu(t) e_k \right)
  \end{equation*}
  where
\[\le\{\begin{array}{ll}
\ds{      \mu \frac{d\hat{f}_k^\mu}{dt}(t) = - \hat{g}_k^\mu(t),} & \ds{\hat{f}_k^\mu(0) = u_k = \left< u, e_k \right>_H} \\
&\vs
\ds{      \mu \frac{d\hat{g}_k^\mu}{dt}(t) = \mu \alpha_k \hat{f}_k^\mu(t) - \hat{g}_k^\mu(t), } & \ds{\hat{g}_k^\mu(0) = v_k = \left< v, e_k \right>_H.}
\end{array}\r.\]
From these equations we see that
  \begin{equation*}
    \left|\hat{g}_k^\mu(t) \right|^2 = - \frac{\mu^2 \alpha_k}{2} \frac{d}{dt} \left| \hat{f}_k^\mu(t) \right|^2 - \frac{\mu}{2} \frac{d}{dt} \left|\hat{g}_k^\mu(t) \right|^2.
  \end{equation*}
 This means that
  \begin{equation}
  \label{m2-1}
  \begin{array}{l}
\ds{    \left|(L^\mu_{t_1,t_2})^\star z \right|_{L^2((t_1,t_2);H)}^2 }\\
\vs
    \displaystyle{ = \frac{1}{2} \sum_{k=1}^\infty \left( \frac{\lambda_k^2}{\alpha_k} |u_k|^2 + \frac{\lambda_k^2}{\mu \alpha_k^2} |v_k|^2  - \frac{\lambda_k^2}{\alpha_k} \left| \hat{f}_k^\mu(t_2-t_1) \right|^2 - \frac{\lambda_k^2}{\mu \alpha_k^2} \left| \hat{g}_k^\mu(t_2 -t_1) \right|^2  \right)}\\
    \vs
    \displaystyle{ = \frac{1}{2} \left( \left| C_\mu^{1/2} z \right|_\H^2 - \left| C_\mu^{1/2} S_\mu^\star(t_2-t_1)z \right|_\H^2 \right).}
  \end{array}
  \end{equation}
Notice that
  \begin{equation*}
    (1 \wedge \sqrt{\mu}) \left| C_\mu^{1/2} z \right|_\H \leq \left|C_1^{1/2} z \right|_\H \leq (1 + \sqrt{\mu}) \left| C_\mu^{1/2} z \right|_\H
  \end{equation*}
  and that $C_1^{1/2}$ commutes with $S_\mu^\star(t)$.
Therefore, by \eqref{negative-type-eq}
  \begin{equation*}
    \left|C_\mu^{1/2} S_\mu^\star(t) z \right|_\H
    \leq \frac1{1\wedge \sqrt{\mu}} \left| S_\mu^\star(t) C_1^{1/2} z \right|_\H
    \leq  \frac{1+\sqrt{\mu}}{1\wedge \sqrt{\mu}}M_\mu  e^{-\omega_\mu t} \left| C_\mu^{1/2} z \right|_\H.
  \end{equation*}
According to \eqref{m2-1}, this implies
  \begin{equation*}
    \left|(L^\mu_{t_1,t_2})^\star z \right|_{L^2((t_1,t_2);H)}^2 \geq \frac{1}{2} \left( 1 -  \le(\frac{1+\sqrt{\mu}}{1\wedge \sqrt{\mu}\,}\r)^2M_\mu^2 e^{-2 \omega_\mu (t_2 -t_1)} \right) \left|C_\mu^{1/2} z \right|_\H^2,
  \end{equation*}
  so that, if we take $T_\mu>0$ such that
  \[\le(\frac{1+\sqrt{\mu}}{1\wedge \sqrt{\mu}\,}\r)^2M_\mu^2 e^{-2 \omega_\mu T_\mu}<1\]
  we can conclude.
\begin{flushright}
$\Box$
\end{flushright}

\section{The main result}

  If $z \in\,C((-\infty,0];\H)$ is such that $I^\mu_{-\infty,0}(z) < +\infty$,  then we have
  \begin{equation} \label{I-mu-explicit-eq}
    I^\mu_{-\infty}(z) = \frac{1}{2} \int_{-\infty}^0 \left| Q^{-1} \left( \mu \frac{\partial^2 \varphi}{\partial t^2}(t) + \frac{\partial \varphi}{\partial t}(t) - A \varphi(t) + Q^2 D\!F(\varphi(t)) \right) \right|_H^2 dt.
  \end{equation}
  Actually, if $I^\mu_{-\infty,0}(z) < +\infty$, then there exists $\psi \in\,L^2((-\infty,0);H)$ such that
  $\varphi= \Pi_1 z$ is a weak solution to
  \begin{equation*}
    \mu \frac{\partial^2 \varphi}{\partial t^2} (t) = A \varphi(t) - \frac{\partial \varphi}{\partial t}(t) - Q^2 D\!F(\varphi(t)) + Q \psi.
  \end{equation*}
This means that
  \begin{equation*}
    \psi(t) = Q^{-1} \left( \mu \frac{\partial^2 \varphi}{\partial t^2}(t) + \frac{\partial \varphi}{\partial t}(t) - A \varphi(t) + Q^2 D\!F(\varphi(t))  \right)
  \end{equation*}
  and \eqref{I-mu-explicit-eq} follows.

  By the same argument, if $I_{-\infty,0}(\varphi) < +\infty$, then it follows that
 \begin{equation} \label{I-explicit-eq}
    I_{-\infty}(\varphi) = \int_{-\infty}^0 \left|Q^{-1} \left(\frac{\partial \varphi}{\partial t}(t) - A \varphi(t) + Q^2 D\!F(\varphi(t)) \right) \right|_H^2 dt
  \end{equation}

\begin{Theorem}
\label{m2-10}
 For any fixed $\mu>0$ and $(x,y) \in\,D((-A)^{1/2}Q^{-1})\times D(Q^{-1})$ it holds
  \begin{equation}
    V^\mu(x,y) = \left|(-A)^{\frac{1}{2}}Q^{-1} x \right|_H^2 + 2F(x) + \mu \left|Q^{-1} y \right|_H^2.
  \end{equation}
  Moreover
  \begin{equation}
    V(x) = \left|(-A)^{\frac{1}{2}} Q^{-1} x \right|_H^2 + 2F(x).
  \end{equation}
In particular,  for any $\mu>0$,
  \begin{equation*}
    V_\mu(x): = \inf_{y \in H^{-1}} V^\mu(x,y) = V^\mu(x,0) = V(x).
  \end{equation*}
\end{Theorem}

\begin{proof}
  First, we observe that if $\varphi(t) = \Pi_1 z(t)$, then
  \begin{equation} \label{I-mu-rewrite-eq}
\begin{array}{l}
\ds{I^\mu_{-\infty}(z) = \frac{1}{2} \int_{-\infty}^0 \left|Q^{-1} \left( \mu \frac{\partial \varphi}{\partial t}(t) - \frac{\partial \varphi}{\partial t} (t) - A \varphi(t) + Q^2 D\!F(\varphi(t))  \right) \right|_H^2 dt }\\
\vs
\ds{+ 2 \int_{-\infty}^0 \left< Q^{-1}\frac{\partial \varphi}{\partial t} (t), Q^{-1} \left(\mu \frac{\partial^2 \varphi}{\partial t^2} (t) - A \varphi(t) \right) + Q D\!F(\varphi(t)) \right>_H dt.}
  \end{array}
  \end{equation}
Now, if
\[\lim_{t \to -\infty} |C_\mu^{-1/2} z(t)|_\H =0,\]
  then
  \begin{equation*}
    \lim_{t \to -\infty} \left|(-A)^{\frac{1}{2}} Q^{-1} \varphi(t) \right|_H + \left|Q^{-1} \frac{\partial \varphi}{\partial t} (t) \right|_H = 0,
  \end{equation*}
so that
\[  \begin{array}{l}
 \ds{2 \int_{-\infty}^0 \left< Q^{-1}\frac{\partial \varphi}{\partial t} (t), Q^{-1} \left(\mu \frac{\partial^2 \varphi}{\partial t^2} (t) - A \varphi(t) \right) + Q D\!F(\varphi(t)) \right>_H dt}\\
\vs
\ds{= \left| (-A)^{\frac{1}{2}} Q^{-1} \varphi(0) \right|_H^2 + 2 F(\varphi(0)) + \mu \left|Q^{-1} \frac{\partial \varphi}{\partial t}(0) \right|_H^2.}
  \end{array}\]
This yields
\[V^\mu(x,y) \ge \left| (-A)^{\frac{1}{2}} Q^{-1} x \right|_H^2 + 2 F(x) + \mu \left|Q^{-1}  y \right|_H^2.\]

Now, let $\tilde{z}(t)$ be a mild solution of the  problem
  \begin{equation*}
   \tilde{z}(t) =  \Smu(t) (x,-y) - \int_0^t \Smu(t-s) \Qmu Q D\!{F}(\tilde{z}(s)) ds,
  \end{equation*}
and let  $(x,y) \in D(C_\mu^{-1/2})$.
  Then $\tilde{\varphi} (t)= \Pi_1 \tilde{z}(t)$ is a weak solution of the problem
  \begin{equation*}
 \mu \frac{\partial^2  \tilde{\varphi}}{\partial t^2}(t) = A \tilde{\varphi}(t) - \frac{\partial  \tilde{\varphi}}{\partial t} (t) - Q^2 D\!F(\tilde{\varphi}(t)),\ \ \ \  \tilde{\varphi}(0)=x,\ \ \ \frac{ \tilde{\varphi}}{\partial t}(0)=-y.
   \end{equation*}
Moreover, as proven below in Lemma \ref{go-to-zero-lemma},
\[\lim_{t \to -\infty} \left| C_\mu^{-1/2} \tilde{z}(t) \right|_\H =0.\]
  Then, if we define $\hat{\varphi}(t) = \tilde{\varphi}(-t)$ for $t \leq 0$,  we see that $\hat{\varphi}(t)$ solves
  \begin{equation*}
    \mu \frac{\partial^2 \hat{\varphi}}{\partial t^2} (t) = A \hat{\varphi}(t) + \frac{\partial \hat{\varphi}}{\partial t} (t) -Q^2 D\!F(\hat{\varphi}(t)),\ \ \ \hat{\varphi}(0)=x,\ \ \ \frac{\partial \hat{\varphi}}{\partial t}(0) = y.
  \end{equation*}
Thanks to \eqref{I-mu-rewrite-eq} this yields
  \begin{equation*}
    I^\mu_{-\infty}(\hat{\varphi}) = \left| (-A)^{\frac{1}{2}} Q^{-1} x \right|_H^2 + 2 F(x) + \mu \left|Q^{-1} y \right|_H^2.
  \end{equation*}
  and then
  \begin{equation*}
    V^\mu(x,y) = \left| (-A)^{\frac{1}{2}} Q^{-1} x \right|_H^2 + 2 F(x) + \mu \left|Q^{-1} y \right|_H^2.
  \end{equation*}

As known, an analogous result holds for $V(x)$.  In what follows, for completeness, we give a proof. We have
\begin{equation}
 \label{I-rewrite-eq}
 \begin{array}{l}
\ds{    I_{-\infty}(\varphi) = \frac{1}{2} \int_{-\infty}^0 \left|Q^{-1} \left( \frac{\partial \varphi}{\partial t} (t) + A \varphi(t) -Q^2 D\!F(\varphi(t))  \right) \right|_H^2 dt } \\
\vs
\ds{+ 2 \int_{-\infty}^0 \left<Q^{-1} \frac{\partial \varphi}{\partial t} (t), Q^{-1} \left( -A \varphi(t) + Q^2 D\!F(\varphi(t)) \right)   \right>_H dt.}
\end{array}
  \end{equation}
From this we see that
  \begin{equation*}
    V(x) \ge \left|(-A)^{\frac{1}{2}} Q^{-1} x \right|_H^2 + 2 F(x).
  \end{equation*}
Just as for the wave equation, for $x \in D((-A)^{\frac{1}{2}} Q^{-1})$, we define $\tilde{\varphi}$ to be the solution of
  \begin{equation*}
    \tilde{\varphi}(t) = e^{tA} x - \int_0^t e^{(t-s)A} Q^2 D\!F(\tilde{\varphi}(s)) ds.
  \end{equation*}
We have
  \begin{equation*}
    \lim_{t \to +\infty}|(-A)^{\frac{1}{2}} Q^{-1}\tilde{\varphi}(t)|_H =0.
  \end{equation*}
Then, if we define $\hat{\varphi}(t) = \tilde{\varphi}(-t)$ we get
  \begin{equation*}
    \frac{\partial \hat{\varphi}}{\partial t} (t) = -A \hat{\varphi}(t) + Q^2 D\!F( \hat{\varphi}(t)),
  \end{equation*}
   so that
   \begin{equation*}
     I_{-\infty}(\hat{\varphi}) = \left|(-A)^{\frac{1}{2}} Q^{-1} x \right|_H^2 + 2 F(x)
   \end{equation*}
   and
   \begin{equation*}
     V(x) = \left|(-A)^{\frac{1}{2}} Q^{-1} x \right|_H^2 + 2 F(x).
   \end{equation*}
\end{proof}
Now, in order to conclude the proof of Theorem \ref{m2-10}, we have to prove the following result.

\begin{Lemma} \label{go-to-zero-lemma}
  Let $(x,y) \in\, D((-A)^{\frac{1}{2}}Q^{-1})\times D(Q^{-1})$ and let $\varphi$ solve the problem
  \begin{equation}
  \label{m2-0}
 \mu \frac{\partial^2 \varphi}{\partial t^2} (t) = A \varphi(t) - \frac{\partial \varphi}{\partial t} (t) - Q^2 D\!F(\varphi(t)),\ \ \ \ \  \varphi(0) = x,\ \
      \frac{\partial \varphi}{\partial t} (0) = y.
        \end{equation}
  Then \[z(t) = \left( \varphi(t), \frac{\partial \varphi}{\partial t}(t) \right) \in\,D((-A)^{\frac{1}{2}})\times D(Q^{-1}),\ \ \ \ t \geq 0,\]
  and
  \begin{equation}
  \label{m2-12}
    \lim_{t \to +\infty} \left|C_1^{-1/2} z(t)  \right|_\H =0.
  \end{equation}
\end{Lemma}

\begin{proof}
  If in \eqref{m2-0} we take the inner product with  $2 Q^{-2} \partial \varphi/\partial t(t)$, we have
  \begin{equation} \label{inner-product-with-varphi-prime}
2\left| Q^{-1} \frac{\partial \varphi}{\partial t} (t) \right|_H^2 = -\frac{d}{dt} \le(\mu\left|Q^{-1} \varphi(t)\right|_H^2 +\left|(-A)^{\frac{1}{2}} Q^{-1} \varphi(t) \right|_H^2 +2F(\varphi(t))\r).
  \end{equation}
Therefore, if we define
  \[
    \Phi_\mu(x,y) = \left|(-A)^{\frac{1}{2}} Q^{-1} x \right|_H^2 + \mu \left| Q^{-1} y \right|_H^2 + 2 F(x),
  \]
as a consequence of \eqref{inner-product-with-varphi-prime} we get
  \begin{equation} \label{energy-decreases-eq}
     \Phi_\mu \left( z(t)  \right) \leq \Phi_\mu(u,v).
  \end{equation}

  Next, by \eqref{m2-0} and the assumption that $\left<DF(x),x \right> \ge 0$, we calculate that
  \[ 
  \begin{array}{l}
  \ds{\frac{d}{dt} \left|Q^{-1} \left(\mu \frac{\partial \varphi}{\partial t}(t) + \varphi(t) \right) \right|_H^2
  = 2 \left<Q^{-1} \left(\mu \frac{\partial \varphi}{\partial t}(t) + \varphi(t) \right), Q^{-1}A \varphi(t) - Q DF(\varphi(t))  \right>_H}\\
  \vs
  \ds{\leq -2 \left|Q^{-1} (-A)^{\frac{1}{2}} \varphi(t) \right|_H^2 - \mu \frac{d}{dt} \left|Q^{-1} (-A)^{\frac{1}{2}} \varphi(t) \right|_H^2 -2\mu \frac{d}{dt} F(\varphi(t)).}
  \end{array}
  \]
  A consequence of this is that
  \begin{equation} \label{varphi-in-L2}
    2 \int_0^\infty \left|Q^{-1} (-A)^{\frac{1}{2}} \varphi(t) \right|_H^2 dt \leq \left|\mu Q^{-1}y + Q^{-1} x \right|_H^2 + \mu\left|Q^{-1} (-A)^{\frac{1}{2}} x \right|_H^2 + 2\mu F(x).
  \end{equation}
Now, if $z(t) = \left(\varphi(t), \frac{\partial \varphi}{\partial t}(t) \right)$,  for any $t,T>0$ we have
  \begin{equation*}
    z(T+t) = \Smu(t) z(T) - \int_T^{T+t} \Smu(T+t -s) \Qmu Q D\!F(\varphi(s)) ds.
  \end{equation*}
By \eqref{negative-type-eq}, and \eqref{m2-00} we have
  \begin{equation*}
  \begin{array}{l} \displaystyle
{    \left| C_1^{-1/2} \int_T^{T+t} \Smu(t+T - s) \Qmu Q D\!F(\varphi(s)) ds \right|_\H} \\
\vs
\displaystyle {\leq  \int_T^{T+t} \left| \Smu(t + T -s) \Qmu (-A)^{\frac{1}{2}} D\!F(\varphi(s)) \right|_\H ds}\\
\vs
    \displaystyle {\leq c\int_T^{T+t} e^{-\omega_\mu( t+T -s)}\left|(-A)^{\frac{1}{2}} Q D\!F(\varphi(s)) \right|_{H^{-1}}ds}\\
    \vs
    \displaystyle {\leq c \int_T^{T+t}  e^{-\omega_\mu (t+T -s)} \left|\varphi(s) \right|_H ds \leq c \left| \varphi \right|_{L^2((T, T+t);H)}.}
  \end{array}
  \end{equation*}
Therefore, by \eqref{varphi-in-L2}, for any $\e>0$ we can pick $T_\e>0$ large enough so that for all $t>0$
  \begin{equation*}
    \left| C_1^{-1/2} \int_{T_\e}^{T_\e+t} \Smu(t+T_\e - s) \Qmu Q D\!F(\varphi(s)) ds \right|_\H  < \frac{\epsilon}{2}.
  \end{equation*}

  Next, by \eqref{negative-type-eq},
  \begin{equation*}
    \left| C_1^{-1/2} \Smu(t) z(T) \right|_\H \leq M_\mu e^{-\omega_\mu t} \left|C_1^{-1/2} z(T) \right|_\H.
  \end{equation*}
Then, as
\[|C_1^{-1/2} z|_\H \leq c\, \Phi_\mu(z),\ \ \ z \in \H,\]  by \eqref{energy-decreases-eq} we can find a $t_\e$ large enough so that for all $T>0$ and $t>t_\e$
  \begin{equation*}
    \left| C_1^{-1/2} \Smu(t) z(T) \right|_\H < \frac{\epsilon}{2}.
  \end{equation*}
   Then for $t> T_\e+ t_\e$
  \begin{equation*}
    \left|C_1^{-1/2} z(t) \right|_\H < \epsilon
  \end{equation*}
  which is what we were trying to prove.
\end{proof}

\end{document}